\documentclass[11pt]{article}
\usepackage{layout}
\usepackage[makeroom]{cancel}
\usepackage{bbm}

\usepackage{amsfonts,dsfont}
\usepackage{amsmath}
\usepackage{array}
\usepackage{amssymb,bm}
\usepackage{amsthm}
\usepackage{graphicx} 
\usepackage{mathtools}
\usepackage{multicol}
\usepackage{mathrsfs}
\usepackage{enumerate,enumitem}
\usepackage[rgb,dvipsnames]{xcolor}
\usepackage{float}
\usepackage[normalem]{ulem}
\usepackage{circuitikz}
\usepackage{cite}

\usepackage{float}

\newtheorem{theorem}{Theorem}[section]

\newtheorem{definition}[theorem]{Definition}

\newtheorem{remark}[theorem]{Remark}
\newtheorem{example}[theorem]{Example}
\newtheorem{examples}[theorem]{Examples}
\newtheorem{foo}[theorem]{Remarks}

\newtheorem{open question}[theorem]{Open Question}
\newtheorem{c/p}[theorem]{Conjecture/Proposition}

\def\vint{\mathop{\mathchoice%
 {\setbox0\hbox{$\displaystyle\intop$}\kern 0.22\wd0%
 \vcenter{\hrule width 0.6\wd0}\kern -0.82\wd0}%
 {\setbox0\hbox{$\textstyle\intop$}\kern 0.2\wd0%
 \vcenter{\hrule width 0.6\wd0}\kern -0.8\wd0}%
 {\setbox0\hbox{$\scriptstyle\intop$}\kern 0.2\wd0%
 \vcenter{\hrule width 0.6\wd0}\kern -0.8\wd0}%
 {\setbox0\hbox{$\scriptscriptstyle\intop$}\kern 0.2\wd0%
 \vcenter{\hrule width 0.6\wd0}\kern -0.8\wd0}}%
 \mathopen{}\int}

\newcommand{\M}{\mathbb M}

\DeclareMathOperator{\Cut}{Cut}

\usepackage[textwidth=3.2cm]{todonotes}

\title{A note on first eigenvalue estimates by coupling methods in K\"ahler and quaternion K\"ahler manifolds}

\author{ Fabrice Baudoin\footnote{F.B. was partly supported by the NSF grant DMS~1901315.}, Guang Yang\footnote{G.Y. was partly supported by the NSF grant DMS~1901315.}, Gunhee Cho}

\date{\today}

\begin{document}

\maketitle

\begin{abstract}
In this short note, using the Kendall-Cranston coupling, we study on K\"ahler (resp. quaternion K\"ahler) manifolds first eigenvalue estimates in terms of dimension, diameter, and lower bounds on the holomorphic (resp. quaternionic) sectional curvature.
\end{abstract}

\tableofcontents

\section{Introduction}

Using probabilistic coupling arguments M.F. Chen and F.Y. Wang \cite{MR1308707,MR1450586} proved the following remarkable theorem on Riemannian manifolds:

\begin{theorem}[Chen-Wang, \cite{MR1308707,MR1450586}]\label{theo_intro}
Let $(\M,g)$ be a compact Riemannian manifold with dimension $n$, diameter $D$ and first eigenvalue $\lambda_1$. Assume that $\mathrm{Ric} \ge (n-1)k$ with $k\in \mathbb{R}$. Then, for any $C^2$ function $g:[0,D) \to \mathbb R$ such that $g(0)=0$ and $g'>0$ on $[0,D)$ one has
\[
\lambda_1 \ge - \sup_{r \in (0,D)} \frac{\mathcal{A}_{k} g(r)}{g(r)},
\]
where
\[
\mathcal{A}_{k}=4 \frac{\partial^2}{\partial r^2}+ (n-1)G(k,r)\frac{\partial}{\partial r},
\]
and $G(k,r)$ is the comparison function \eqref{comparison}.

\end{theorem}
For instance, if $k=0$,  the choice $g(r)=\sin \left(\frac{\pi}{2D} r \right )$ in Theorem \ref{theo_intro}  yields the celebrated Zhong-Yang \cite{MR794292} estimate $\lambda_1 \ge \frac{\pi^2}{D^2}$.

In this note we apply in the context of K\"ahler and quaternion K\"ahler manifolds a similar coupling method to improve this first eigenvalue estimate on K\"ahler and quaternion K\"ahler manifolds. 	We rely on the decomposition of Ricci curvature on K\"ahler manifolds into orthogonal Ricci curvature and holomorphic sectional curvature as follows.
\begin{equation*}
	Ric(X,\overline{X})=Ric^{\perp}(X,\overline{X})+R(X,\overline{X},X,\overline{X})/|X|^2,
\end{equation*}
where  $X$ is  a $(1,0)$-tangent vector of the holomorphic tangent bundle on a K\"ahler manifold $\M^n$. A similar decomposition holds on a quaternion K\"ahler manifold. It seems worthy to mention that to obtain Laplacian and Index comparison theorems with model spaces in K\"ahler (and quaternion K\"ahler) geometry, it is necessary that the lower bounds of orthogonal Ricci curvature and holomorphic (quaternionic) sectional curvature should be assumed simultaneously, not only a Ricci curvature lower bound (also, see \cite{GM21}). In this paper, the main geometric ingredients are new estimates of the index form in those settings that build on the previous recent work \cite{FG20}. 

\

\textbf{Acknowledgement:} Shortly after this work was posted on arXiv, the authors were made aware of the references \cite{XLKW20,XLKW21} which independently obtained comparable results under similar assumptions but  by using analytic methods. The authors thank the anonymous referees for a careful reading and remarks that improved the presentation of the paper.

\section{Preliminaries: K\"ahler and quaternion K\"ahler manifolds}

In this section, for the sake of completeness, we give the definitions we will be using in this paper. We refer to \cite{FG20}  for more details.
Throughout the paper, let $(\M,g)$ be a smooth complete Riemannian manifold. Denote by $\nabla$ the Levi-Civita connection on $\M$.

\subsection{K\"ahler manifolds}

\begin{definition}
	The manifold $(\M,g)$ is called a K\"ahler manifold, if there exists a smooth  $(1,1)$ tensor $J$ on $\M$ that satisfies:
	\begin{itemize}
		\item For every $x \in \M$, and $X,Y \in T_x\M$, $g_x(J_x X,Y)=-g_x(X,J_xY)$;
		\item For every $x \in \M$, $J_x^2=-\mathbf{Id}_{T_x\M} $;
		\item $\nabla J$=0.
	\end{itemize}
	The map $J$ is called a complex structure.
\end{definition}

It is well-known that K\"ahler manifolds can be seen as the complex manifolds for which the Chern connection coincide with the Levi-Civita connection, see \cite{MR2325093}, however the complex viewpoint will not be necessary to state and prove our results, so we will always only consider the real structure on a K\"ahler manifold and work with the definition above.

We will be considering the holomorphic sectional curvature and orthogonal Ricci curvature which are defined as below. A reason to consider those curvature quantities instead of the more usual Ricci curvature is that in the classical Riemannian  comparison theorems involving a Ricci curvature lower bound, the spaces with respect to which the comparison is made are usually spheres, Euclidean spaces and hyperbolic spaces. To develop comparison theorems in the category  of K\"ahler manifolds, it is then necessary to consider more subtle curvature invariants which are adapted to the additional K\"ahler structure. We refer to  \cite{MR3858834} and \cite{MR3996490} for further details about the geometric interpretation of the holomorphic sectional curvature and orthogonal Ricci curvature.

Let
\[
R(X,Y,Z,W)=g ( (\nabla_X \nabla_Y -\nabla_Y \nabla_X -\nabla_{[X,Y]} )Z, W)
\]
be the Riemannian curvature tensor of $(\M,g)$. 
The holomorphic sectional curvature of the K\"ahler manifold $(\M,g,J)$ is defined as
\[
H(X)=\frac{R(X,JX,JX,X)}{g(X,X)^2}.
\]

The orthogonal Ricci curvature  of the K\"ahler manifold $(\M,g,J)$ is defined for a vector field  $X$ such that $g(X,X)=1$ by
\[
\mathrm{Ric}^\perp (X,X)=\mathrm{Ric} (X,X)-H(X),
\]
where $\mathrm{Ric}$ is the usual Riemannian Ricci tensor of $(\M,g)$. Unlike the Ricci tensor, $\operatorname{Ric}^{\perp}$ does not admit  polarization, so we never consider $\operatorname{Ric}^\perp (u, v)$ for $u\not= v$.

 The table below shows the curvature of the K\"ahler model spaces $\mathbb{C}^m$ , $\mathbb{C}P^m$ and $\mathbb{C}H^m$, see \cite{FG20}.

\begin{table}[H]
\centering
\scalebox{0.8}{
\begin{tabular}{|p{1.5cm}||p{1.5cm}|c|>{\centering\arraybackslash}p{1.8cm}|c|  }
  \hline
 $\M$ &  $H$ & $\mathrm{Ric}^\perp$  \\
  \hline
 \hline
$\mathbb{C}^m$ &  $0 $ & $0$ \\
 $\mathbb{C}P^m$ & 4 & $2m-2$  \\
$\mathbb{C}H^m$  &  -4 & $-(2m-2)$  \\
  \hline
\end{tabular}}
\caption{Curvatures of K\"ahler model spaces.}
\label{Table 1}
\end{table}

\subsection{Quaternion K\"ahler manifolds}

\begin{definition}
The manifold $(\M,g)$ is called a  quaternion K\"ahler manifold, if   there exists a covering of $\M$ by open sets $U_i$ and, for each $i$,  3 smooth  $(1,1)$ tensors $I,J,K$ on $U_i$ such that:

\begin{itemize}
\item For every $x \in U_i$, and $X,Y \in T_x\M$, $g_x(I_x X,Y)=-g_x(X,I_xY)$,  $g_x(J_x X,Y)=-g_x(X,J_xY)$, $g_x(K_x X,Y)=-g_x(X,K_xY)$ ;
\item For every $x \in U_i$, $I_x^2=J_x^2=K_x^2=I_xJ_xK_x=-\mathbf{Id}_{T_x\M} $;
\item For every $x \in U_i$, and $X\in T_x\M$,  $\nabla_X I, \nabla_X J, \nabla_X K \in \mathbf{span} \{ I,J,K\}$;
\item For every $x \in U_i \cap U_j$, the vector space of endomorphisms of $T_x\M$ generated by $I_x,J_x,K_x$ is the same for $i$ and $j$.
\end{itemize}
\end{definition}
On quaternion K\"ahler manifolds, we will be considering the following curvatures. Let
\[
R(X,Y,Z,W)=g ( (\nabla_X \nabla_Y -\nabla_Y \nabla_X -\nabla_{[X,Y]} )Z, W)
\]
be the Riemannian curvature tensor of $(\M,g)$. We define the quaternionic sectional  curvature of the quaternionic K\"ahler manifold $(\M,g,J)$ as
\[
Q(X)=\frac{R(X,IX,IX,X)+R(X,JX,JX,X)+R(X,KX,KX,X)}{g(X,X)^2}.
\]

We define the orthogonal Ricci curvature  of the quaternionic K\"ahler manifold $(\M,g,I,J,K)$  for a vector field  $X$ such that $g(X,X)=1$ by
\[
\mathrm{Ric}^\perp (X,X)=\mathrm{Ric} (X,X)-Q(X),
\]
where $\mathrm{Ric}$ is the usual Riemannian Ricci tensor of $(\M,g)$. The table below shows the curvature of the quaternion-K\"ahler model spaces $\mathbb{H}^m$ , $\mathbb{H}P^m$ and $\mathbb{H}H^m$, see \cite{FG20}.

 \begin{table}[H]
\centering
\scalebox{0.8}{
\begin{tabular}{|p{1.5cm}||p{1.5cm}|c|>{\centering\arraybackslash}p{1.8cm}|c|  }
  \hline
 $\M$ &   $Q$ & $\mathrm{Ric}^\perp$  \\
  \hline
 \hline
$\mathbb{H}^m$ &  $0 $ & $0$ \\
 $\mathbb{H}P^m$ & 12 & $4m-4$  \\
$\mathbb{H}H^m$  &  -12 & $-(4m-4)$  \\
  \hline
\end{tabular}}
\caption{Curvatures of the quaternion K\"ahler model spaces.}
\label{Table 2}
\end{table}

\section{Index form estimates on K\"ahler and quaternion K\"ahler manifolds}

Let $(\M,g)$ be a complete Riemannian manifold with real dimension $n$ and denote by $d$ the Riemannian distance on $\M$.  The index form of a vector field $X$ (with not necessarily vanishing endpoints) along a geodesic $\gamma$ is defined by
\begin{align*}
I (\gamma,X,X): = \int_0^T \left( \langle \nabla_{\gamma'}  X, \nabla_{\gamma'} X  \rangle-\langle R(\gamma',X)X,\gamma'\rangle \right) dt,
\end{align*}
where $\nabla$ is the Levi-Civita connection and $R$ is the Riemann curvature tensor of $\M$.

We will denote by 
\[
\Cut (\M)=\left\{ (x,y) \in \M \times \M, \, x \in \Cut (y) \right\}.
\]
where $\Cut (y)$ denotes the cut-locus of $y$. For $(x,y) \notin \Cut (\M)$ we denote
\[
\mathcal{I}(x,y)=\sum_{i=1}^{n-1} I( \gamma, Y_i,Y_i)
\]
where $\gamma$ is the unique length parametrized geodesic from $x$ to $y$ and $\{ Y_1, \cdots, Y_{n-1} \}$ are Jacobi fields such that at both $x$ and $y$, $\{ \gamma', Y_1, \cdots, Y_{n-1} \}$ is an orthonormal frame.

Throughout the paper, we consider the comparison function:

\begin{equation}\label{comparison}
G(k,r) = \begin{cases}  -2 \sqrt{ k} \tan \frac{\sqrt{k}r}{2} & \text{if $k > 0$,} \\
0 & \text{if $k = 0$,}\\ 2 \sqrt{|k|} \tanh \frac{\sqrt{|k|}r}{2} & \text{if $k < 0$.} \end{cases}
\end{equation}

\subsection{K\"ahler case}

Let $(\M,g,J)$ be a complete K\"ahler with complex dimension $m$ (i.e. the real dimension is $2m$). As above, the holomorphic sectional curvature of $\M$ will be denoted by $H$ and the orthogonal Ricci curvature by $\mathrm{Ric}^\perp$.


\begin{theorem}\label{comparison index  kahler}
Let $k_1,k_2 \in \mathbb{R}$. Assume that $H \ge 4k_1$ and that $\mathrm{Ric}^\perp \ge (2m-2)k_2$. For every $(x,y) \notin \Cut (\M)$, one has
\[
\mathcal{I}(x,y) \le (2m-2)G(k_2,d(x,y)) +2 G(k_1,2d(x,y)).
\]
\end{theorem}

\begin{remark}
In particular, if $k_1>0$ then $\M$ is compact with a diameter $\le \frac{\pi}{2\sqrt{k_1}}$ and if $k_2 >0$ then $\M$ is compact with a diameter $\le \frac{\pi}{\sqrt{k_2}}$.
\end{remark}

\begin{proof}
Let $(x,y) \notin \Cut (\M)$ and denote by $\gamma:[0,r] \to \M$ where $r=d(x,y)$ the unique length parametrized  geodesic from $x=\gamma(0)$ to $y=\gamma(r)$. At $x$, we consider an orthonormal frame $\{ X_1(x),\cdots, X_{2m}(x) \}$ such that
\[
X_1(x) =\gamma'(0), \, X_2(x)=J \gamma'(0).
\]
We denote by $ X_1,\cdots, X_{2m} $ the vector fields obtained by parallel transport of $ X_1(x),\cdots, X_{2m}(x) $ along $\gamma$. Note that $X_1 =\gamma'$ and that $X_2 =J\gamma'$ because $\nabla J=0$.

We introduce the function
\[
\mathfrak{j}(k,t) = \mathfrak{c}(k,t) +\frac{1- \mathfrak{c}(k,r)}{ \mathfrak{s}(k,r)}   \mathfrak{s}(k,t) 
 \]
where
\begin{equation*}
\mathfrak{s}(k,t) = \begin{cases}   \sin  \sqrt{k} t & \text{if $k > 0$,} \\
t & \text{if $k = 0$,}\\  \sinh \sqrt{|k|} t & \text{if $k < 0$,} \end{cases}
\end{equation*}

and

\begin{equation*}
\mathfrak{c}(k,t) = \begin{cases}   \cos  \sqrt{k} t & \text{if $k > 0$,} \\
1 & \text{if $k = 0$,}\\  \cosh \sqrt{|k|} t & \text{if $k < 0$.} \end{cases}
\end{equation*}

We now consider the vector fields defined along $\gamma$ by $\tilde{X}_2(\gamma(t))=\mathfrak{j}(4k_1,t) X_2$ and for $i=3,\cdots,2m$ by $\tilde{X}_i(\gamma(t))=\mathfrak{j}(k_2,t) X_i.$ From the index lemma (Lemma 1.21 in \cite{MR2394158}) one has $\mathcal{I}(x,y)  \le \sum_{i=2}^{2m} I( \gamma, \tilde{X}_i ,\tilde{X}_i).$
We first estimate
\begin{align*}
I( \gamma, \tilde{X}_2 ,\tilde{X}_2)&=\int_0^r \left( \langle \nabla_{\gamma'}  \tilde{X}_2, \nabla_{\gamma'} \tilde{X}_2  \rangle-\langle R(\gamma',\tilde{X}_2)\tilde{X}_2,\gamma'\rangle \right) dt \\
 &\le \int_0^r (\mathfrak{j}'(4k,t)^2+4k \mathfrak{j}(4k,t)^2)dt=2 G(k_1,2r).
\end{align*}
Then, by a similar computation one has $\sum_{i=3}^{2m} I( \gamma, \tilde{X}_i ,\tilde{X}_i) \le (2m-2)G(k_2,r)$.
\end{proof}

\subsection{Quaternion K\"ahler case}

Let now $(\M,g,I,J,K)$ be a complete quaternion K\"ahler manifold with quaternionic dimension $m$ (i.e. the real dimension is $4m$). As above, the quaternionic sectional curvature of $\M$ will be denoted by $H$ and the orthogonal Ricci curvature by $\mathrm{Ric}^\perp$.


\begin{theorem}\label{comparison index  quaternion kahler}
Let $k_1,k_2 \in \mathbb{R}$. Assume that $Q \ge 12k_1$ and that $\mathrm{Ric}^\perp \ge (4m-4)k_2$.  For every $(x,y) \notin \Cut (\M)$,
\[
\mathcal{I}(x,y) \le (4m-4)G(k_2,d(x,y)) +6 G(k_1,2d(x,y)).
\]
\end{theorem}

\begin{remark}
In particular, if $k_1>0$ then $\M$ is compact with a diameter $\le \frac{\pi}{2\sqrt{k_1}}$ and if $k_2 >0$ then $\M$ is compact with a diameter $\le \frac{\pi}{\sqrt{k_2}}$.
\end{remark}
\begin{proof}
The proof is almost similar to the K\"ahler case. We can assume $m \ge 2$ since for $m=1$ theorem \ref{comparison index  quaternion kahler} reduces to theorem \ref{theo_intro}. Let $(x,y) \notin \Cut (\M)$ and denote by $\gamma:[0,r] \to \M$ where $r=d(x,y)$ the unique length parametrized  geodesic from $x=\gamma(0)$ to $y=\gamma(r)$. At $x$, we consider an orthonormal frame $\{ X_1(x),\cdots, X_{4m}(x) \}$ such that
\[
X_1(x) =\gamma'(0), \, X_2(x)=I \gamma'(0), \, X_3(x)=J \gamma'(0), \, X_4(x)=K \gamma'(0)
\]
We denote by $ X_1,\cdots, X_{4m} $ the vector fields obtained by parallel transport of $ X_1(x),\cdots, X_{4m}(x) $ along $\gamma$ and consider  the vector fields defined along $\gamma$ by
\[
\tilde{X}_2(\gamma(t))=\mathfrak{j}(4k_1,t) X_2, \, \tilde{X}_3(\gamma(t))=\mathfrak{j}(4k_1,t) X_3, \tilde{X}_4(\gamma(t))=\mathfrak{j}(4k_1,t) X_4
\]
and for $i=5,\cdots,4m$ by
\[
\tilde{X}_i(\gamma(t))=\mathfrak{j}(k_2,t) X_i.
\]
 Since along $\gamma$ one has
\[
\nabla_{\gamma'} I, \nabla_{\gamma'} J, \nabla_{\gamma'} K \in \mathbf{span} \{ I,J,K\} 
\]
we deduce that along $\gamma$ one has
\[
\mathbf{span} \{ X_2, X_3, X_4 \}= \mathbf{span} \{ I \gamma' ,J \gamma' , K \gamma' \}.
\]
Moreover $\{ X_2, X_3, X_4 \}$ and $\{ I \gamma' ,J \gamma' , K \gamma' \}$ are both orthonormal along $\gamma$. One deduces
\begin{align*}
 & R(\gamma',X_2,X_2,\gamma')+R(\gamma',X_3,X_3,\gamma')+R(\gamma',X_4,X_4,\gamma') \\
=& R(\gamma',I\gamma',I\gamma',\gamma')+R(\gamma',J\gamma',J\gamma',\gamma')+R(\gamma',K\gamma',K\gamma',\gamma') \\
=& Q (\gamma').
\end{align*}

Using then the index lemma as in the K\"ahler case and similar arguments as in the proof of theorem \ref{comparison index  kahler}   yields the conclusion.
\end{proof}

\section{First eigenvalue estimates}

With  the index form estimates of the previous section in hands, we can use the reflection coupling method by M.F. Chen and F.Y. Wang \cite{MR1308707,MR1450586}  (see also \cite[Section 6.7]{Hsu}) to get curvature diameter estimates of the first eigenvalue in K\"ahler and quaternion K\"ahler manifolds.

\subsection{K\"ahler case}

Let $(\M,g,J)$ be a compact K\"ahler with complex dimension $m$. We denote by $D$ the diameter of $\M$ and by $\lambda_1$ the first eigenvalue of $\M$.

\begin{theorem}\label{main kahler}
Let $k_1,k_2 \in \mathbb{R}$. Assume that $H \ge 4k_1$ and that $\mathrm{Ric}^\perp \ge (2m-2)k_2$. For any $C^2$ function $g:[0,D] \to \mathbb R$ such that $g(0)=0$ and $g'>0$ on $[0,D)$ one has
\[
\lambda_1 \ge - \sup_{r \in (0,D)} \frac{\mathcal{L}_{k_1,k_2} g(r)}{g(r)},
\]
where
\[
\mathcal{L}_{k_1,k_2}=4 \frac{\partial^2}{\partial r^2}+ \left( (2m-2)G(k_2,r) +2 G(k_1,2r)\right)\frac{\partial}{\partial r}.
\]
\end{theorem}

\begin{proof}
Let $g:[0,D] \to \mathbb R$ be a $C^2$ function such that $g(0)=0$ and $g'>0$ on $[0,D)$. As in \cite{MR1308707,MR1450586} we will use the coupling by reflection introduced by Kendall \cite{kendall}. 

The  argument is easy to explain in the absence of cut-locus. Indeed, for the sole sake of the explanation, assume first that $\mathrm{Cut}(\M)=\emptyset$. From Theorem 2.3.2 in \cite{MR3154951} and Theorem \ref{comparison index  kahler},  if  $(X_t,Y_t)_{t \ge 0}$ is  the Kendall's  mirror coupling started from $(x,y)$ then, by denoting $\rho_t=d(X_t,Y_t)$ we have that:
\[
 d \rho_t  \le 2 d\beta_t+\left( (2m-2)G(k_2,\rho_t) +2 G(k_1,2\rho_t)\right) dt,
\]
where $(\beta_t)_{t \ge 0}$ is a Brownian motion on $\mathbb R$ with $\langle \beta \rangle_t=2t$. 
Using then It\^o's formula and the fact that $g$ is non decreasing, we obtain 
\begin{align*}
 g(\rho_t) &  \le g(\rho_0) +2\int_0^t g'(\rho_s) d\beta_s+4\int_0^t g''(\rho_s) ds+\int_0^t \left( (2m-2)G(k_2,\rho_s) +2 G(k_1,2\rho_s)\right) g'(\rho_s) ds \\\
  &\le  g(\rho_0) +2\int_0^t g'(\rho_s) d\beta_s+\int_0^t \mathcal{L}_{k_1,k_2} g(\rho_s) ds \\
  &\le  g(\rho_0) +2\int_0^t g'(\rho_s) d\beta_s-\delta \int_0^t  g(\rho_s) ds,
\end{align*}
where $\delta=- \sup_{r \in (0,D)} \frac{\mathcal{L}_{k_1,k_2} g(r)}{g(r)}$. By taking expectations, we obtain therefore
\[
\mathbb{E}( g(\rho_t)) \le \mathbb{E}( g(\rho_0))-\delta \int_0^t  \mathbb{E}(g(\rho_s)) ds.
\]
Thanks to Gronwall's inequality this yields
\[
\mathbb{E}( g(\rho_t)) \le \mathbb{E}( g(\rho_0)) e^{-\delta t}.
\]

On the other hand, let us now consider an eigenfunction $\Phi$ associated to the eigenvalue $\lambda_1$, i.e. $\Delta \Phi =-\lambda_1 \Phi$. One has
\begin{align*}
e^{-\lambda_1t } | \Phi(x)-\Phi(y)|& = | \mathbb{E} ( \Phi(X_t)-\Phi(Y_t) )| \\
 & \le \mathbb{E} (| \Phi(X_t)-\Phi(Y_t) |) \\
 & \le \| \nabla \Phi \|_\infty \mathbb{E}( d(X_t,Y_t)) = \| \nabla \Phi \|_\infty \mathbb{E}( \rho_t).
\end{align*}
Thanks to our assumptions on $g$, there exists a constant $C>0$ so that for every $r \in [0,D]$, $g(r) \ge \frac{1}{C} r$. We obtain therefore
\begin{align*}
e^{-\lambda_1t } | \Phi(x)-\Phi(y)| \le C  \| \nabla \Phi \|_\infty  \mathbb{E}( g(\rho_t)) \le C  \| \nabla \Phi \|_\infty g( d(x,y) ) e^{-\delta t}.
\end{align*}
Since it is true for every $t \ge 0$ one concludes that $\lambda_1 \ge \delta$.

In the absence of cut locus issues the Kendall coupling is easily constructed using stochastic differential equations and the above argument is complete. To handle cut-locus issues, it is possible to instead construct the Kendall coupling $(X_t,Y_t)$ as a limit of coupled random walks, see \cite{kuwada}  and \cite{MR2080605}.  In particular, a similar argument as in \cite[Lemma 11]{kuwada} yields as above
\[
\mathbb{E}( g(\rho_t)) \le \mathbb{E}( g(\rho_0))-\delta \int_0^t  \mathbb{E}(g(\rho_s)) ds.
\]
and the conclusion follows then as before.
\end{proof}

\begin{remark}
Theorem \ref{main kahler} can be used to improve the lower bound of Theorem \ref{theo_intro} in some situations like for example  the complex projective space $\mathbb{C}P^m$ in Table \eqref{Table 1}.  Indeed $H \ge 4k_1$ and  $\mathrm{Ric}^\perp \ge (2m-2)k_2$ imply the Ricci lower bound $\mathrm{Ric} \ge 4k_1+(2m-2)k_2$ and we always have
\[
(2m-2)G(k_2,r) +2 G(k_1,2r) \le (2m-1) G \left( \frac{4k_1+(2m-2)k_2}{2m-1} , r \right).
\]
by concavity of $k \to G(k,r)$.
\end{remark}
%

%
%
%
%
\subsection{Quaternion K\"ahler case}

Let now $(\M,g,I,J,K)$ be a complete quaternion K\"ahler with quaternionic dimension $m$. As before, we denote by $D$ the diameter of $\M$ and by $\lambda_1$ the first eigenvalue of $\M$. Using the coupling by reflection as in the previous section, we obtain the following result.

\begin{theorem}\label{main Q-kahler}
Let $k_1,k_2 \in \mathbb{R}$. Assume that $Q \ge 12k_1$ and that $\mathrm{Ric}^\perp \ge (4m-4)k_2$. For any $C^2$ function $g:[0,D) \to \mathbb R$ such that $g(0)=0$ and $g'>0$ on $[0,D)$ one has
\[
\lambda_1 \ge - \sup_{r \in (0,D)} \frac{\tilde{\mathcal{L}}_{k_1,k_2} g(r)}{g(r)},
\]
where
\[
\tilde{\mathcal{L}}_{k_1,k_2}=4 \frac{\partial^2}{\partial r^2}+ \left( (4m-4)G(k_2,r) +6 G(k_1,2r)\right)\frac{\partial}{\partial r}.
\]
\end{theorem}

\begin{remark}
As in the K\"ahler case, Theorem \ref{main Q-kahler}  can be used to improve Theorem \ref{theo_intro} in some situations like the quaternionic projective space $\mathbb{H}P^m$ in Table \eqref{Table 2}. Indeed $Q \ge 12k_1$ and $\mathrm{Ric}^\perp \ge (4m-4)k_2$ imply the Ricci lower bound $\mathrm{Ric} \ge 12k_1+(4m-4)k_2$ and we have
\[
(4m-4)G(k_2,r) +6 G(k_1,2r) \le (2m-1) G \left( \frac{12k_1+(4m-4)k_2}{4m-1} , r \right).
\]
\end{remark}

\bibliographystyle{amsplain}
\bibliography{references}


\end{document}